\documentclass[12pt]{article}

\usepackage[english]{babel}

\usepackage{enumerate}

\usepackage{fullpage} 
\usepackage{parskip} 

\usepackage{setspace}
\usepackage{amsmath, amssymb, amsfonts, latexsym, stmaryrd, enumerate, comment}
\usepackage{amsthm}
\usepackage{tikz, tikz-cd, tikzsymbols}
\usepackage[english]{babel}
\usepackage{graphicx}
\usepackage{sidecap}
\usepackage{xcolor}
\usepackage{caption}
\usepackage[normalem]{ulem}
\usepackage{multicol}
\usepackage{xr}
\usepackage[colorlinks=true, allcolors=blue]{hyperref}

\usepackage{mathdots}
\usepackage{mathtools}
\usetikzlibrary{decorations.pathreplacing}

\newtheorem{theorem}{Theorem}

\newtheorem{cor}[theorem]{Corollary}
\newtheorem{lemma}[theorem]{Lemma}

\newtheorem{proposition}[theorem]{Proposition}

\newtheorem*{conj*}{Conjecture}
\theoremstyle{definition}
\newtheorem{definition}[theorem]{Definition}

\numberwithin{theorem}{section}
\numberwithin{equation}{section}
\numberwithin{figure}{section}

\usepackage{hyperref}

\hypersetup{
  colorlinks   = true,
  urlcolor     = blue, 
  linkcolor    = blue,
  citecolor   = red
}

\newcommand{\J}{J}
\newcommand{\Jmax}{\J_{\max}}
\newcommand{\Jmin}{\J_{\min}}

\newcommand{\newG}{G^*}
\newcommand{\newH}{H^*}
\newcommand{\newJ}{J^*}

\newcommand{\newnewG}{G^{**}}

\newcommand{\newnewJ}{J^{**}}

\newcommand{\newdeg}{\deg^*}
\newcommand{\tree}[2]{\mathcal{T}_{#1,#2}}

\newcommand{\broom}[2]{B_{#1,#2}}

\newcommand{\lever}[2]{\mathcal{L}_{#1,#2}}

\newcommand{\Tn}{\mathcal{T}_n}
\newcommand{\Tnd}{\tree{n}{d}}
\newcommand{\Anr}{\mathcal{A}_{n,r}}

\newcommand{\Bnd}{B_{n,d}}
\newcommand{\Dnd}{D_{n,d}}
\newcommand{\Lnd}{L_{n,d}}

\def\Tmeet{T_{\mathrm{meet}}}
\def\Tbestmeet{T_{\mathrm{bestmeet}}}

\tikzstyle{vertex}=[draw,thick,fill=white,circle,inner sep=2pt]
\tikzstyle{full}=[draw,thick,fill=black,circle,inner sep=2pt]
\tikzstyle{bubble}=[draw,thick,fill=white,circle,x radius=10pt,y radius=10pt]
\tikzstyle{tiny}=[draw,thick,circle,inner sep=1.5pt]

\usetikzlibrary{decorations.pathreplacing}

\title{Random Walks and the Best Meeting Time for Trees}
\author{Andrew Beveridge\footnote{Department of Mathematics, Statistics and Computer Science, Macalester College,  Saint Paul, MN, USA, \texttt{abeverid@macalester.edu}} ~and Ari Holcombe Pomerance\footnote{Department of Mathematical and Statistical Sciences, University of Colorado Denver, Denver, CO, USA, \texttt{ari.holcombepomerance@ucdenver.edu}}}
\date{}

\begin{document}
\maketitle

\begin{abstract}
We consider random walks on a tree $G=(V,E)$ with stationary distribution $\pi_v = \deg(v)/2|E|$ for $v \in V$. Let the hitting time $H(v,w)$ denote the expected number of steps required for the random walk started at vertex $v$ to reach vertex $w$. We characterize the extremal tree structures for the best meeting time
$\Tbestmeet(G) = \min_{w \in V} \sum_{v \in V} \pi_v H(v,w)$ for trees of order $n$ with diameter $d$. The best meeting time is maximized by the balanced double broom graph, and it is minimized by the balanced lever graph.
\end{abstract}

\section{Introduction}

Let $G= (V,E)$ be a connected graph.  A \emph{random walk} on $G$ starting at vertex $w$ is a sequence of vertices $w=w_0, w_1, \ldots, w_t, \ldots$ such that  for $t \geq 0$, we have  
$$
\Pr(w_{t+1} = v \mid w_t = u ) =
\left\{ 
\begin{array}{cl}
1/\deg(u) & \mbox{if } (u,v) \in E, \\
 0  & \mbox{otherwise}.
 \end{array}
 \right.
 $$
See \cite{Haggstrom2002,Lovasz1996} for introductions to random walks on graphs, and see \cite{MPL2017,RM2021} for surveys of contemporary random walk applications.

The  \emph{hitting time} $H(u,v)$ is the expected number of steps before a random walk started at vertex $u$ reaches vertex $v$. We define $H(u,u)=0$ for the case $u=v$. 
When $G$ is not bipartite, 
 the distribution of $w_t$ converges to the \emph{stationary distribution} $\pi$, given by $\pi_v = \deg(v)/2|E|$. For bipartite $G$, we have convergence when we follow a \emph{lazy walk}, which remains at the current vertex with probability $1/2$. Lazy walks simply double the hitting times, so we will consider non-lazy walks for simplicity.

Extremal graph structures for random walks on trees have garnered interest in recent decades \cite{BW1990, BW2013, Zhang2022, Li2022}. We contribute to this effort by considering the meeting time at $v \in V$, which  is the expected hitting time to target vertex $v$ from a source vertex selected randomly according to the stationary distribution $\pi$. 

\begin{definition}
For a graph $G=(V,E)$ and a vertex $v \in V$, the \emph{meeting time at} $v$ is 
$$
H(\pi, v) = \sum_{u \in V} \pi_u H(u,v).
$$
The \emph{meeting time of graph} $G$ is 
$$
\Tmeet(G) = \max_{v \in V} H(\pi, v)    
$$
and the \emph{best meeting time of graph} $G$ is 
$$
\Tbestmeet(G) = \min_{v \in V} H(\pi, v).     
$$
\end{definition}

The meeting time is also called the global mean first-passage time \cite{TBV2009} and the (random) walk centrality \cite{XZ2024} for the vertex. The meeting time has been studied on various graph families, including regular lattices \cite{Montroll}, random graph models \cite{Kittas2008} and fractal lattices \cite{HR2008}. The best meeting time is of particular interest, since the corresponding vertex can be considered to be the most central one.
Herein, we study the best meeting time for random walks on trees of order $n$ with diameter $d$. 

 \begin{definition}
     The family of trees of order $n$  is denoted by $\Tn$.
     The family of trees of order $n$ with diameter $d$ is denoted by $\Tnd$.
 \end{definition}

We denote the path on $n$ vertices by $P_n$ and the star on $n$ vertices by $S_n = K_{1,n-1}$. We encounter three further subfamilies of trees.

\begin{definition}
\label{def:lever}
A \emph{lever} $G \in \tree{n}{d}$ 
 consists of a path $v_0, \ldots, v_{d}$ with $n-d-1$ pendant edges incident with one \emph{fulcrum vertex} $v_k$ where $2 \leq k \leq n-1$. 
 The \emph{balanced lever} $L_{n,d} \in \Tnd$ is the lever with fulcrum $v_{\lfloor d/2 \rfloor}$.
\end{definition}

\begin{definition}
\label{def:broom}
For $3 \leq d <n$, the broom $B_{n,d} \in \tree{n}{d}$  consists of a path $v_1, \ldots, v_{d}$ with  leaves $w_1, \ldots w_{n-d}$ incident with $v_1$, where we also label $v_0=w_1$ for convenience. The path $v_1, \ldots, v_{d}$ is the \emph{handle} of the broom, while $w_1, \ldots, w_{n-d}$ are the \emph{bristles} of the broom.     
\end{definition}

\begin{definition}
\label{def:double-broom}
A \emph{double broom} $G \in \tree{n}{d}$  consists of a path $v_1, \ldots, v_{d-1}$ with $\ell \geq 1$ (resp. $r \geq 1$) pendant edges incident with $v_1$ (resp. $v_{d-1}$), where we label one of these leaves as $v_0$ (resp.~$v_d$). 
We define $\Dnd \in \Tnd$ to be the \emph{balanced double broom} which satisfies $\ell = \lfloor (n-d-1)/2 \rfloor$ and  $r = \lceil(n-d-1)/2 \rceil$. 
\end{definition}

Figure \ref{fig:lever-broom-double} offers examples of a balanced lever, a broom, and a balanced double broom.

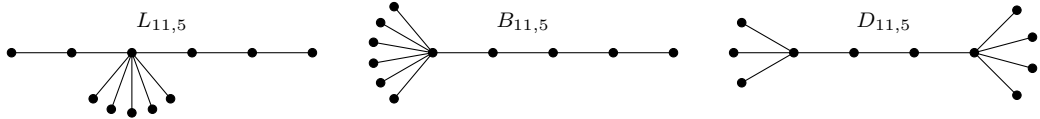
\begin{figure}
    \begin{center}
        \begin{tikzpicture}[scale=.8]

\begin{scope}

\draw (0,0) -- (5,0);

\foreach \x in {0,1,2,3,4,5}
{
\draw[fill] (\x,0) circle (2pt);
}

\begin{scope}[shift={(2,0)}]

\foreach \x in {-50,-70,-90,-110,-130}
{
\draw (0,0) -- (\x:1);
\draw[fill] (\x:1) circle (2pt);
}

\end{scope}    

\node at (2.5, .5) {\scriptsize $L_{11,5}$};

\end{scope}

\begin{scope}[shift={(6,0)}]

\draw (1,0) -- (5,0);

\foreach \x in {1,2,3,4,5}
{
\draw[fill] (\x,0) circle (2pt);
}

\begin{scope}[shift={(1,0)}]

\foreach \x in {130,150,170,190,210,230}
{
\draw (0,0) -- (\x:1);
\draw[fill] (\x:1) circle (2pt);
}
\end{scope}    

\node at (2.5, .5) {\scriptsize $B_{11,5}$};

\end{scope}

\begin{scope}[shift={(12,0)}]

\draw (0,0) -- (4,0);

\foreach \x in {0,1,2,3,4}
{
\draw[fill] (\x,0) circle (2pt);
}

\begin{scope}[shift={(1,0)}]

\foreach \x in {150,210}
{
\draw (0,0) -- (\x:1);
\draw[fill] (\x:1) circle (2pt);
}
\end{scope}  

\begin{scope}[shift={(4,0)}]

\foreach \x in {45,15,-15,-45}
{
\draw (0,0) -- (\x:1);
\draw[fill] (\x:1) circle (2pt);
}
\end{scope}  

\node at (2.5, .5) {\scriptsize $D_{11,5}$};

\end{scope}
    
\end{tikzpicture}
    \end{center}
    \caption{The balanced lever $L_{11,5}$, the broom $B_{11,5}$ and the balanced double broom $D_{11,5}$.}
    \label{fig:lever-broom-double}
\end{figure}

 Recent work has explored the extremal properties of meeting times on trees. Kemeny's constant
 $$ \kappa(G) = \sum_{u \in V} \sum_{v \in V} \pi_u \pi_v H(u,v) = \sum_{v \in V} \pi_v H(\pi,v)$$ 
 averages the meeting times of all vertices. 
Upper and lower bounds on $\kappa(G)$ for trees  $G \in \tree{n}{d}$ were  developed in \cite{CDK2020}: the unique minimizer is the balanced lever $L_{n,d}$ and the unique maximizer is the balanced double broom graph $\Dnd$.  

Meanwhile, upper and lower bounds on $\Tmeet(G)$ for $G \in \Tnd$ were determined in \cite{BBHP}. The unique maximizer is the broom graph $\Bnd$, and the unique minimizer is the balanced double broom $\Dnd$ or a slight variant, depending on the relative parities of $n$ and $d$.
Complementary to \cite{BBHP}, we characterize the best meeting time, providing lower and upper bounds on $\Tbestmeet(G)$ for $G \in \Tnd$. The
the extremal trees for the best meeting time 
are the same as those for Kemeny's constant.

\begin{theorem}
\label{thm:min-bestmeet}
For $2 \leq d \leq n-1$ the quantity $\min_{G \in \mathcal{T}_{n,d}} \Tbestmeet(G)$ is achieved uniquely by the balanced lever graph $L_{n,d}$. We have
\begin{equation}
\label{eqn:min-bestmeet}
\Tbestmeet(L_{n,d})
= 
\begin{cases}
\frac{d^3-d}{6(n-1)} + \frac{1}{2} &  \mbox{if $d$ is odd,} \\
\frac{d^3-4d}{6(n-1)} + \frac{1}{2} & \mbox{if $d$ is even}. \\
\end{cases}
\end{equation}

\end{theorem}

\begin{theorem}
\label{thm:max-bestmeet}
For $2 \leq d \leq n-1$, the quantity 
$\max_{G \in \mathcal{T}_{n, d}}  \Tbestmeet(G)$
is achieved uniquely by the balanced double broom $D_{n,d}$. We have
\begin{equation}
\label{eqn:max-bestmeet}
\Tbestmeet(\Dnd) =
\begin{cases}
\frac{1}{2} \left((d-2) n -d^2 +3 d-2 \right) +\frac{d^3-6 d^2+11 d-6}{6 (n-1)} 
& \mbox{if $n$ odd and $d$ odd,} \\
\frac{1}{2} \left( (d-2) n -d^2 +3 d - 1 \right) +\frac{d^3-6 d^2+8 d}{2(n-1)}
& \mbox{if $n$ odd and $d$ even,} \\
 \frac{1}{2} \left( (d-2) n - d^2 + 3 d \right) +\frac{d^3-6 d^2+8 d}{6 (n-1)}
 & \mbox{if $n$ even and $d$ odd,} \\
\frac{1}{2}\left( (d-2) n- d^2 + 3 d - 1 \right) +\frac{d^3-6 d^2+11 d-6}{6 (n-1)}
 & \mbox{if $n$ even and $d$ even.} \\
\end{cases}
\end{equation}
\end{theorem}

We aggregate our results for diameters $2 \leq d \leq n-1$ to obtain bounds on the best meeting time of trees of order $n$. We have an unexpected result for odd $n \geq 9$: the broom $B_{n,n-2}$ 
 is the maximizing structure, rather than the path $P_n$. 

\begin{theorem}
\label{thm:best-meet-trees}
The extremal best meeting times for $G \in \Tn$ are as follows. 
We have
$$
\min_{G \in \Tn} \Tbestmeet(G) = \Tbestmeet(S_n) = \frac{1}{2}.
$$

If $n$ is even then  
$$
\max_{G \in \Tn} \Tbestmeet(G) = \Tbestmeet(P_n) =\frac{1}{6} ( n^2 -2n +3).
$$
If $n$ is odd then
$$
\max_{G \in \Tn} \Tbestmeet(G) = 
\begin{cases}
    \Tbestmeet(P_n) = \frac{1}{6}(n^2-2n) & \mbox{if $3 \leq n \leq 7$}, \\
    \Tbestmeet(B_{n,n-2}) = \frac{1}{6}(n^2-2n+3) - \frac{4}{(n-1)} & \mbox{if $n \geq 9$}. \\
\end{cases}
$$
\end{theorem}

\section{Preliminaries}

\subsection{The joining time}

The acyclic, connected structure of trees leads to straightforward hitting time formulas. For example, it is well-known that the path $P_{d+1}$ on vertices $v_0, v_1, \ldots, v_d$, has hitting times
\begin{equation*}
\label{eqn:path-hitting-time}    
    H(v_i, v_j) = \begin{cases}
        j^2 - i^2 & 0 \leq i \leq j \leq d,\\ 
        (d - j)^2 - (d - i)^2 & 0 \leq j \leq i \leq d.
    \end{cases}
\end{equation*}
More generally, equation (2.6) in \cite{beveridge2009} states that for an arbitrary tree $G=(V,E)$, we have
\begin{equation}
\label{eqn:hit-time}
H(u,w) = \sum_{v \in V} \ell(u,v;w) \deg(v)
\end{equation}
where 
$$
\ell(u,v;w) = \frac{1}{2} ( d(u,w) + d(v,w) - d(u,v))
$$
is the length of the intersection of the $(u,w)$-path and the $(v,w)$-path. Fundamentally, all of the hitting time results in this paper rely on equation \eqref{eqn:hit-time}.

Throughout this work, it will be convenient to scale the  meeting time  $\sum_{u \in V} \pi_u H(u,v)$ by $2|E|=2(n-1)$. This clears the denominator, which simplifies calculations that decompose the meeting time with respect to subgraphs of $G$. With that in mind, we make the following definition. 

\begin{definition}
Let $G=(V,E)$ be a graph. For a vertex $w \in V$,
the \emph{joining time to} $w$ is
\begin{equation}
\label{eqn:scaled-pi-to-vertex}
J(v) = 2|E| H(\pi,v)  = \sum_{u \in V} \deg(u) H(u,v).
\end{equation}
The \emph{maximum joining time of $G$} is
$$
\Jmax(G) = \max_{w \in V} J(w)
$$
and the \emph{minimum joining time of $G$} is
$$
\Jmin(G) = \min_{w \in V} J(w).
$$
\end{definition}

For example, we calculate the maximum join time for the path $P_n$ and the star $S_n$ using the following result from \cite{BW2013} about the meeting times of these graphs.

\begin{theorem}[Theorem 1.2 of \cite{BW2013}]
\label{thm:path-pi-to-endpoint}
For the path $P_n$ on vertices $v_0, v_1, \dots, v_{n-1}$, we have
$$
\Tmeet(P_n) = H(\pi, v_{n-1}) = \frac{4n^2-8n+3}{6} = \frac{2}{3}(n-1)^2 - \frac{1}{6}.
$$
For the star $S_n$ with center $v_0$ and leaf vertices $v_1, \ldots, v_{n-1}$, we have
$$
\Tmeet(S_n) = H(\pi,v_{n-1}) = 2n - \frac{7}{2}.
$$
\end{theorem}

The corresponding maximum joining times are
\begin{equation}
\label{eqn:path-join-time}   
\Jmax(P_n) 
= \frac{4}{3}(n-1)^3-\frac{1}{3}(n-1)
=  \frac{1}{3} \left( 4 n^3-4 n^2 + 11 n \right) -1.
\end{equation}
and
$$
\Jmax(S_n) = 2n^2 -\frac{11}{2}n + \frac{7}{2}.
$$

\subsection{The barycenter}

There are multiple ways to define centers for random walks on trees \cite{beveridge2009}. We will be interested in the barycenter, which is the ``average center'' of the tree.

\begin{definition}
The \emph{barycenter} of a tree $G=(V,E)$ is the vertex (or two adjacent vertices) achieving
$\min_{v \in V} \sum_{w \in V} d(v,w)$.
\end{definition}

The next proposition reveals that the barycenter plays a central role for random walks of the tree. 

\begin{proposition}[Proposition 1.1 of \cite{beveridge2009}]
\label{prop:barycenter-v2}
Let $G=(V,E) \in \Tn$ be a tree of order $n$.  The following statements for $c \in V$ are equivalent.
\begin{enumerate}[(a)]
    \item Vertex $c$ is a barycenter of $G$.
    \item Vertex $c$ satisfies $H(v,c) \leq H(c,v)$ for every $v \in V$.
    \item $\J(c) = \min_{v \in V} \J(v)$.
    \item For each component $H$ of $G-c$, we have $|V(H)| \leq n/2$.
\end{enumerate}
\end{proposition}

Statement $(c)$ is essential for our results, since it states that $\Jmin(G) = \min_{v \in V} J(v)$ is achieved by the barycenter of the tree. Statement $(d)$ also plays a critical role: we often use this condition to identify the barycenter $c$.


\subsection{Splitting a tree}

During our calculations, we will frequently decompose a tree into subtrees in the following manner.

\begin{definition}
Given a tree $G$ and a vertex $v \in V$ where $\deg(v) = k \geq 2$, the \emph{$v$-split of $G$} is the collection $ \{ G_1, G_2, \ldots, G_k \}$ of subtrees where $G_i$ is the tree induced by $v$ and the $i$th component of $G-v$.    
\end{definition}

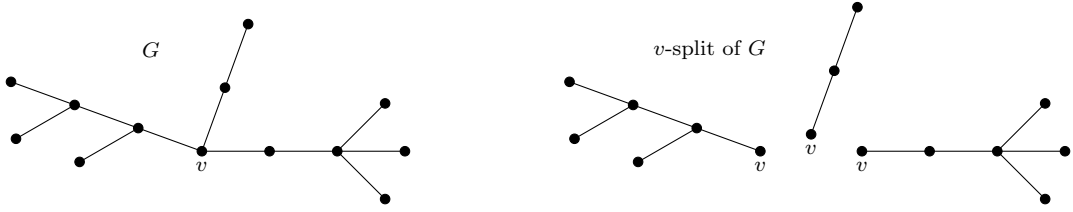
\begin{figure}
\begin{center}
        \begin{tikzpicture}[scale=0.9]

\begin{scope}
    
\draw (0,0) -- (3,0);

\foreach \x in {0,1,2,3}
{
\draw[fill] (\x,0) circle (2pt);
}

\draw (0,0) -- (160:3);

\draw[fill] (160:1) circle (2pt);
\draw[fill] (160:2) circle (2pt);
\draw[fill] (160:3) circle (2pt);

\begin{scope}[shift={(160:1)}]
\draw (0,0) -- (210:1);
\draw[fill] (210:1) circle (2pt);
\end{scope}

\begin{scope}[shift={(160:2)}]
\draw (0,0) -- (210:1);
\draw[fill] (210:1) circle (2pt);
\end{scope}

\begin{scope}[shift={(2,0)}]
\draw (0,0) -- (45:1);
\draw (0,0) -- (-45:1);
\draw[fill] (45:1) circle (2pt);
\draw[fill] (-45:1) circle (2pt);
\end{scope}

\draw (0,0) -- (70:2);
\draw[fill] (70:1) circle (2pt);
\draw[fill] (70:2) circle (2pt);    

\node[below] at (0,0) {\scriptsize $v$};

\node at (-0.75,1.5) {\scriptsize $G$};

\end{scope}

\begin{scope}[shift={(9,0)}]

\begin{scope}[shift={(0.75,0)}]    
\draw (0,0) -- (3,0);

\foreach \x in {0,1,2,3}
{
\draw[fill] (\x,0) circle (2pt);
}

\begin{scope}[shift={(2,0)}]
\draw (0,0) -- (45:1);
\draw (0,0) -- (-45:1);
\draw[fill] (45:1) circle (2pt);
\draw[fill] (-45:1) circle (2pt);
\end{scope}

\node[below] at (0,0) {\scriptsize $v$};

\end{scope}

\begin{scope}[shift={(-0.75,0)}] 

\draw (0,0) -- (160:3);
\draw[fill] (0,0) circle (2pt);

\draw[fill] (160:1) circle (2pt);
\draw[fill] (160:2) circle (2pt);
\draw[fill] (160:3) circle (2pt);

\node[below] at (0,0) {\scriptsize $v$};

\begin{scope}[shift={(160:1)}]
\draw (0,0) -- (210:1);
\draw[fill] (210:1) circle (2pt);
\end{scope}

\begin{scope}[shift={(160:2)}]
\draw (0,0) -- (210:1);
\draw[fill] (210:1) circle (2pt);

\end{scope}

\node at (-.75,1.5) {\scriptsize $v$-split of $G$};

\end{scope}


\begin{scope}[shift={(0,.25)}]
\draw (0,0) -- (70:2);
\draw[fill] (70:1) circle (2pt);
\draw[fill] (70:2) circle (2pt);    
\draw[fill] (0,0) circle (2pt);
\node[below] at (0,0) {\scriptsize $v$};

\end{scope}

\end{scope}

\end{tikzpicture}
\end{center}

    \caption{A tree $G$ and its $v$-split into $\deg(v)=3$ subgraphs.}
    \label{fig:split}
\end{figure}

Figure \ref{fig:split} shows a graph $G$ and its $v$-split. Splitting a tree at vertex $v$ induces a natural decomposition of  the joining time to $v$. Using $J_H(v)$ to denote the joining time to $v$ in subgraph $H$, it is clear that
\begin{equation}
\label{eqn:join-decomp}
    J(v) = \sum_{i=1}^k J_{G_i}(v).
\end{equation}


\section{Minimizing the best meeting time}

In this section, we prove Theorem \ref{thm:min-bestmeet}, that $\min_{G \in \Tnd} \Tbestmeet(G)$ is achieved uniquely by the balanced lever graph $L_{n,d}$.

\subsection{The minimum joining time for paths and levers}

We start by calculating the minimum joining time for paths and levers.

\begin{lemma}
\label{lemma:path-min-join-time}
Let $P_{n}$ be the path of diameter $n-1$ on vertices $v_0, v_1, v_2, \ldots, v_{n-1}$. Then for odd $n$,
$$
\Jmin(P_n) = \min_{k} J(v_k) 
= \frac{1}{3} ( n^3-3n^2+2 n)
$$
is achieved by $k=(n-1)/2$. For even $n$,
$$
\Jmin(P_n) = \min_{k} J(v_k)  = \frac{1}{3} ( n^3-3n^2+5n) - 1
$$
is achieved by $k=(n-2)/2$ and $k=n/2$.
\end{lemma}

\begin{proof}
Taking the $v_k$-split of $P_n$ decomposes the path into 
 into a left-hand $(k+1)$-path $\newG$ on $v_0, v_1, \ldots, v_k$ and a right-hand $(n-k)$-path $\newnewG$ on $v_k, v_{k+1}, \ldots, v_{n-1}$. Using  the joining time formula \eqref{eqn:path-join-time} for each subpath, we have
\begin{align}
\nonumber
\J(v_k) &= \newJ(v_k) + \newnewJ(v_k) \\
\nonumber
&= \left(\frac{4}{3}k^3 - \frac{1}{3}k \right) +
\left(\frac{4}{3}(n-k-1)^3 - \frac{1}{3}(n-k-1) \right) 
\\
\label{eqn:path-min-join}
&= \frac{4}{3} (k^3 + (n-k-1)^3) - \frac{1}{3}(n-1).
\end{align}
Treating this formula 
as a continuous function of $k$, we achieve the minimum value at $k=(n-1)/2$. This is an integer when $n$ is odd. When $n$ is even, both  $k=\lfloor (n-1)/2 \rfloor=(n-2)/2$ and $k=\lceil (n-1)/2 \rceil=n/2$  achieve the minimum for the given path. Plugging these values into equation \eqref{eqn:path-min-join} gives the joining time formulas in the lemma.
\end{proof}

\begin{lemma}
\label{lemma:lever-join-time} 
Let $\lever{n}{d}$ denote the set of lever graphs on $n$ vertices with diameter $2 \leq d \leq n-1$. 
For a lever $G \in \lever{n}{d}$, let $v_0, v_1, \ldots, v_d$ be its geodesic, and let $x=v_k$ be its fulcrum.
For odd $d$, 
$$
\min_{G \in \lever{n}{d}} J(x) = n - 1 + \frac{1}{3} (d^3-d) 
$$
is achieved by the balanced lever $L_{n,d}$ with fulcrum $x = v_{(d-1)/2}$. For even $d$, 
$$
\min_{G \in \lever{n}{d}} J(x) = n - 1 +  \frac{1}{3} (d^3-4d) 
$$ 
is achieved by the balanced lever $L_{n,d}$ with fulcrum $x = v_{(d-2)/2}$.
\end{lemma}

\begin{proof}
We handle $d=2$ and $d=n-1$ exceptionally. Observe that the star $S_n = L_{n,2}$ satisfies $\Jmin(S_n) = n-1$, which matches the value of the formula for $d=2$. At the other extreme,
$P_n=L_{n,n-1}$, and when $d=n-1$, the statement reduces algebraically to Lemma \ref{lemma:path-min-join-time}.

Let $G \in \lever{n}{d}$ where $3 \leq d \leq n-2$. Using $\J_H(v)$ to denote the joining time to $v$ for graph $H$, we have
$$
\J_{G}(x) = \J_{G}(v_k) = \J_{P_{d+1}}(v_k) + (n-d-1).
$$
The lemma follows directly from Lemma \ref{lemma:path-min-join-time}, substituting $n=d+1$, and algebraic simplification.
\end{proof}

\begin{cor}
\label{cor:lever-bestmeet-formula}
For $2 \leq d \leq n-1$, the best meeting time for the balanced lever $\Lnd \in \Tnd$ is
$$
\Tbestmeet(\Lnd) = 
\begin{cases}
\frac{d^3-d}{6(n-1)} + \frac{1}{2} &  \mbox{if $d$ is odd,} \\
\frac{d^3-4d}{6(n-1)} + \frac{1}{2} & \mbox{if $d$ is even}. \\
\end{cases}
$$
\end{cor}

\begin{proof}
Observe that the fulcrum $x$ of the balanced lever $\Lnd$ is the barycenter by 
Proposition \ref{prop:barycenter-v2}(d). By Proposition \ref{prop:barycenter-v2}(c),
$$\frac{\J(x)}{2(n-1)} = \frac{\Jmin(\Lnd)}{2(n-1)} =\Tbestmeet(\Lnd).$$
The equation for $\Tbestmeet(\Lnd)$  follows directly from the equations of the lemma. 
\end{proof}

It is illuminating to consider a sequence of balanced levers $\{ L_{n,d(n)} \}$ for increasing values of $n$, where $d=d(n)$ is a function of $n$. Using asymptotic notation, we have
$$
\min_{G \in \Tnd} \Tmeet(G) = 
\begin{cases}
    \frac{1}{2} + o(1) & \mbox{when } d \mbox{ is constant}, \\
     \frac{c^3+3}{6}  + o(1) & \mbox{when } d(n) = cn^{1/3} + o(n^{1/3}) \mbox{ where } c>0, \\
     \frac{(d(n))^3}{6n} + o((d(n))^3/n) & \mbox{when } \omega(n^{1/3}) = d(n) = o(n).  \\
     \frac{c^3}{6} n^2 + o(n^2) & \mbox{when } d(n) =  cn + o(n) \mbox{ where } 0 < c \leq 1.  \\
\end{cases}
$$

\subsection{Proof of Theorem \ref{thm:min-bestmeet}}

We prove Theorem \ref{thm:min-bestmeet} using a constructive lemma that alters any tree  $G \in \Tnd$ until we arrive at the balanced lever $L_{n,d}$. With each change, the joining time strictly decreases.
When comparing a graph $G$ to another graph $\newG$, we use starred notation for quantities related to $\newG$. For example, $\newH(u,v)$ is a hitting time for $\newG$, 
and $\newJ(v)$ is a joining time for $\newG$, and so on.

We need the following lemma from \cite{BBHP} for relocating a leaf to be adjacent to the target vertex.

\begin{lemma}[Lemma 4.13 in \cite{BBHP}]
\label{lemma:move-leaf}
Let $G=(V,E)$ be a tree.
Let $x,y,z  \in V$ where $z$ is a leaf adjacent to $y$.
Let $\newG = G - (y, z) + (x, z)$.
Then the following hold.
\begin{enumerate}[(a)]
    \item $\newH(v, x) \leq H(v, x)$ for every vertex $v \in V$.
    \item $\newJ(x) < \J(x)$.
\end{enumerate}
\end{lemma}

\begin{lemma}
\label{lemma:min-min-join}
For $2 \leq d \leq n-1$ the quantity
$\min_{G \in \Tnd} \Jmin(G)$
is achieved uniquely by 
the balanced lever graph $L_{n,d}$.
\end{lemma} 

\begin{proof}
The statement is trivially true for $d=2$ and $d=n-1$. So assume that $3 \leq d \leq n-2$. 
Suppose  that $G \in \Tnd$ is not a lever whose barycenter is the fulcrum vertex of degree $n-d+1$. We follow a three-phase process to create such a lever $\newG \in \Tnd$ where 
\begin{equation*}
\Jmin(\newG) =  \min_{v \in V} \newJ(v) < \min_{v \in V} J(v) = \Jmin(G).   
\end{equation*}
Figure \ref{fig:minmin}. shows an example of the three-phase process.
Fix a geodesic $ P = \{v_0, v_1, \ldots, v_d \}$. Let $c \in V(G)$ be a barycenter of $G$, so that $J(c) = \Jmin(G) $ by Proposition \ref{prop:barycenter-v2}. 

\begin{figure}
    \begin{center}
       \begin{tikzpicture}[scale=.5, rotate=-45]

\draw (0,0) -- (0,2);

\foreach \x in {0,1}
{
\draw[fill] (0,\x) circle (2pt);
}

\node[right] at (0,0) {\scriptsize $c$};
\node[right] at (-5,2) {\scriptsize $v_0$};
\node[right] at (0,2) {\scriptsize $v_5$};
\node[right] at (7,2) {\scriptsize $v_{12}$};

\draw (-5,2) -- (7,2);

\foreach \x in {-5,-4,-3,-2,-1,0,1,2,3,4,5,6,7}
{
\draw[fill] (\x,2) circle (2pt);
}

\draw (4,2) -- (4,0);
\draw[fill] (4,1) circle (2pt);
\draw[fill] (4,0) circle (2pt);

\begin{scope}[shift={(-3,2)}]
\draw (0,0) -- (-70:1);
\draw (0,0) -- (-110:1);
\draw[fill] (-70:1) circle (2pt);  
\draw[fill] (-110:1) circle (2pt);  
\end{scope}


\foreach \x in {-170, -150, -130,-110,-90,-70, -50, -30, -10}
{
\draw (0,0) -- (\x:2);
\draw[fill] (\x:1) circle (2pt);
\draw[fill] (\x:2) circle (2pt);
}

\draw[very thick, -latex] (0,3.5) -- (2,5.5);
\node at (0.5,4.75) {\scriptsize Phase One};


\begin{scope}[shift={(5,5)}]

\draw (0,0) -- (0,2);

\foreach \x in {0,1}
{
\draw[fill] (0,\x) circle (2pt);
}

\node[right] at (-5,2) {\scriptsize $v_0$};
\node[right] at (0,2) {\scriptsize $v_5$};
\node[right] at (7,2) {\scriptsize $v_{12}$};



\draw (-5,2) -- (7,2);

\foreach \x in {-5,-4,-3,-2,-1,0,1,2,3,4,5,6,7}
{
\draw[fill] (\x,2) circle (2pt);
}


\foreach \x in {1,...,22}
{
\draw (0,0) -- (100+\x*15:1);
\draw[fill] (100+\x*15:1) circle (2pt);

}

\node at (0.5,4.75) {\scriptsize Phase Two};
\draw[very thick, -latex] (0,3.5) -- (2,5.5);

\end{scope}


\begin{scope}[shift={(10,10)}]

\node[right] at (-5,2) {\scriptsize $v_0$};
\node[right] at (7,2) {\scriptsize $v_{12}$};


\draw (-5,2) -- (7,2);

\foreach \x in {-5,-4,-3,-2,-1,0,1,2,3,4,5,6,7}
{
\draw[fill] (\x,2) circle (2pt);
}


\begin{scope}[shift={(0,2)}]

\foreach \x in {1,...,12}
{
\draw (0,0) -- (12+\x*12:1);
\draw[fill] (12 + \x*12:1) circle (2pt);
\draw (0,0) -- (-12-\x*12:1);
\draw[fill] (-12 - \x*12:1) circle (2pt);
}
\end{scope}

\node at (0.5,4.75) {\scriptsize Phase Three};
\draw[very thick, -latex] (0,3.5) -- (2,5.5);

\end{scope}


\begin{scope}[shift={(15,15)}]

\node[right] at (-5,2) {\scriptsize $v_0$};
\node[right] at (7,2) {\scriptsize $v_{12}$};


\draw (-5,2) -- (7,2);

\foreach \x in {-5,-4,-3,-2,-1,0,1,2,3,4,5,6,7}
{
\draw[fill] (\x,2) circle (2pt);
}


\begin{scope}[shift={(1,2)}]

\foreach \x in {1,...,12}
{
\draw (0,0) -- (12+\x*12:1);
\draw[fill] (12 + \x*12:1) circle (2pt);
\draw (0,0) -- (-12-\x*12:1);
\draw[fill] (-12 - \x*12:1) circle (2pt);
}
\end{scope}

\end{scope}
    
\end{tikzpicture}
    \end{center}
    \caption{The three-phase process that minimizes $\Jmin(G)$. We start with a tree $G \in \tree{37}{12}$ whose barycenter is not on a geodesic path. Phase One: repeatedly move all leaves besides $v_0,v_{12}$ to be adjacent to $c$. Phase Two: Move all non-path vertices to become leaves adjacent to $v_5$. Phase Three: move the fulcrum of the lever from $v_5$ to $v_6$.}
    \label{fig:minmin}
\end{figure}
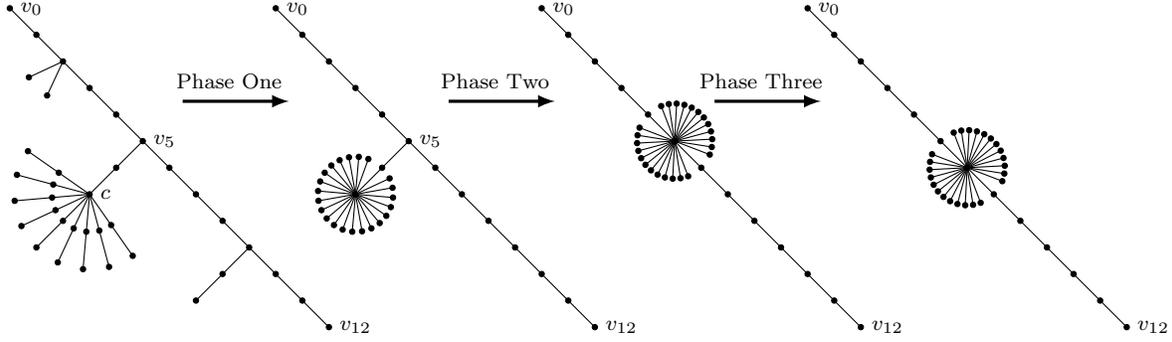

{\bf Phase One.} If $G$ is a lever whose fulcrum is also its barycenter, then move to Phase Three. If every leaf $z \in V \backslash \{ v_0, v_d \}$ is adjacent to the barycenter, then move to Phase Two. Otherwise, there is a leaf $z \in V \backslash \{ v_0, v_d \}$ that is not adjacent to the barycenter $c$. Let $y$ be the unique neighbor of $z$. Define $\newG = G - (z,y) + (z,c)$. Vertex $c$ is  also the barycenter of $\newG$ by Proposition \ref{prop:barycenter-v2}(d), and therefore $\min_{v \in V} \newJ(v) = \newJ(c)$.  
Furthermore, Lemma \ref{lemma:move-leaf}(b) guarantees that  $\newJ(c) < J(c)$.

We repeatedly move leaves in $V \backslash \{v_0, v_d\}$ until they are all are adjacent to the barycenter $c$. If the updated graph is a lever, then move to Phase Three; otherwise continue to Phase Two. 

{\bf Phase Two.} The (updated) tree $G$ has a particular structure because $c \notin V(P)$. There is a unique vertex $v_k \in V(P)$ with $\deg(v_k)=3$. The tree contains a $(v_k, c)$-path $Q$, and the remaining vertices $V(G \backslash (P \cup Q))$ are leaves adjacent to barycenter $c$.  

Let $\newG$ be the lever obtained by turning every vertex not on the path $P$ into a leaf adjacent to vertex $v_k$. We claim that $v_k$ is the barycenter of $\newG$. Indeed, Proposition \ref{prop:barycenter-v2}(d) held for vertex $c$ in $G$, so it also holds for vertex $v_k$ in $\newG$. Setting $A = V(P)$ and $B= V \backslash ( A \cup \{c \})$, we partition
\begin{align*}
\newJ(v_k) &= \sum_{v \in V} \newdeg(v) \newH(v,v_k) \\
&=
\sum_{v \in A} \newdeg(v) \newH(v,v_k) + \sum_{v \in B} \newdeg(v) \newH(v,v_k) + \newdeg(c) \newH(c,v_k),
\end{align*}
and we bound each of these three terms. We have
$$ 
\sum_{v \in A} \newdeg(v) \newH(v,v_k) = \sum_{v \in A} \deg(v) H(v,v_k)  
$$
and
$$
\sum_{v \in B} \newdeg(v) \newH(v,v_k)
= |B| \leq  \sum_{v \in B} \deg(v) H(v,c) 
$$
and (trivially)
$$
\newdeg(c) \newH(c,v_k) = 1 < \sum_{v \in A } \deg(v) H(v_k,c).
$$
Adding these inequalities together yields
\begin{align*}
\sum_{v \in V} \newdeg(v) \newH(v,v_k) 
&<
\sum_{v \in A} \deg(v) (H(v,v_k) +  H(v_k,c))
+ \sum_{v \in B} \deg(v) H(v,c) \\
&=\sum_{v \in A \cup B} \deg(v) H(v,c) = \sum_{v \in V} \deg(v) H(v,c),
\end{align*}
which confirms that $\newJ(v_k) < J(c)$.

{\bf Phase Three.} Our (updated) tree $G$ is a lever whose barycenter $c=v_k$ is the fulcrum vertex. By Lemma \ref{lemma:path-min-join-time}, the unique best choice for this barycenter is the center of the path. So if $v_k$ is not a central vertex, then 
$ J(L_{n,d}) < J(G)$.
\end{proof}

\begin{proof}[Proof of Theorem \ref{thm:min-bestmeet}]
Lemma \ref{lemma:min-min-join} shows that for $G \in \Tnd$, we have
$\Jmin(L_{n,d}) \leq \Jmin(G)$, with equality if and only if $G=L_{n,d}$ is the balanced lever. Therefore $L_{n,d}$ is the unique minimizer of $\Tbestmeet(G)$ for $G \in \Tnd$. 
The formula for $\Tbestmeet(L_{n,d})$ was established in Corollary \ref{cor:lever-bestmeet-formula}.
\end{proof}


\section{Maximizing the best meeting time}

In this section, we prove Theorem \ref{thm:max-bestmeet}, that $\max_{G \in \Tnd} \Tbestmeet(G)$ is achieved uniquely by the balanced double broom $\Dnd$. We require the following lemmas from \cite{BBHP}.

\begin{lemma}[Lemma 3.4 in \cite{BBHP}]
\label{lemma:broom-join-max}
Consider the broom graph $\Bnd$, where $d \geq 2$.
Then
\begin{equation}
\label{eqn:join-broom-tip}    
\Jmax(\Bnd) = J(v_d) =  
4 (d-1) n^2  +(5  -4 d^2) n+ \frac{4 d^3 -4 d-3}{3}.
\end{equation}
\end{lemma}

\begin{lemma}[Lemma 3.9 in \cite{BBHP} ]
\label{lemma:broomify}
Let $\Anr$ be the set of rooted trees $(G,z)$, where $z \in V$, on $n$ vertices such that $r = \max_{v \in V} d(v,z)$.
The quantity
$
\max_{(G,z) \in \Anr}  \J(z)
$ 
is achieved uniquely by the broom graph $B_{n, r}$ where $z=v_r$ is the final leaf of the broom handle. 
\end{lemma}



\subsection{The minimum joining time for balanced double brooms}

We show that among the double brooms $G \in \Tnd$, the balanced double broom is the unique maximizer of $\Jmin(G)$. We then develop a formula for the minimum joining time for balanced double brooms. 

\begin{lemma}
\label{lemma:balanced-double-broom}
Let $\mathcal{D}_{n,d}$ denote the family of double brooms on $n$ vertices with diameter $d$. Then
$
\max_{G \in \mathcal{D}_{n, d}}  \Jmin(G)
$
is achieved uniquely by a barycenter of the balanced double broom $D_{n,d}$.
\end{lemma}

\begin{proof}
We must handle two cases, depending on the location of the barycenter $c$.

{\bf Case $c \notin \{ v_1, v_{d-1} \}$.} Take the $c$-split of $G$ to create brooms $G_1$ and $G_2$.
Let $d_1$ be the diameter of $G_1$ and let $d_2 = d-d_1$ be the diameter of $G_2$.

{\bf Subcase $n=2k+1$ is odd.}
For $i=1,2$,  each subtree $G_i$ contains $k+1$ vertices (including $c$), so  $G_i$ is the broom $B_{k+1,d_i}$ where $c$ is the final handle vertex. By equation \eqref{eqn:join-broom-tip}, the join time for $G$ is
$$
J(c) = \sum_{i=1}^2 \left( 4(d_i-1) (k+1)^2 + (5-4d_i^2)(k+1) + \frac{4d_i^3-4d_i-3}{3} \right).
$$
To find the maximal value, we
set $d_2 = d-d_1$ and then differentiate with respect to $d_1$. This gives
$$
(8k-4d+8 ) (d - 2 d_1)
$$
which is zero when $d_1=d/2$. Therefore when $d$ is even, taking $d_1=d/2$ is optimal; when $d$ is odd, taking $d_1=(d-1)/2$ is optimal. 

{\bf Subcase $n=2k$ is even.}
Without loss of generality,  $G_1$ is the broom $\broom{k}{d_1}$ and $G_2$ is the broom $\broom{k+1}{d_2} =\broom{k}{d-d_1}$. 
 By equation \eqref{eqn:join-broom-tip}, the join time for $G$ is
\begin{align*}
J(c) & = 4(d_1-1) k^2 + (5-4d_1^2)k + \frac{4d_1^3-4d_1-3}{3} \\
& \qquad +
4(d_2-1) (k+1)^2 + (5-4d_2^2)(k+1) + \frac{4d_2^3-4d_2-3}{3}    
\end{align*}
To find the maximal value, we
set $d_2 = d-d_1$ and then differentiate with respect to $d_1$. This gives
$$
(8k -4d+4)(d-1 - 2 d_1)
$$
which is zero when $d_1=(d-1)/2$. Therefore, when $d$ is odd, taking $d_1=(d-1)/2$ is optimal; when $d$ is even, taking $d_1=d/2$ is optimal. 

This completes the proof when $c \notin \{ v_1, v_{d-1} \}$.

{\bf Case $c \in \{ v_1, v_{d-1} \}$.} Without loss of generality, assume that $c=v_{d-1}$. In this case, the $c$-split graph would consist of a broom graph $G_1=B_{m,d-1}$ where $m \leq n/2$, and $r$ subgraphs of order $2$. Instead, we take $G_2$ to be the union of these single edges, which results in a star $S_{r+1}$ with center $v_{d-1}$. Observe that $J_{G_2}(v_{d-1})=r$ because this star has $r$ leaves.

Here is where we are in luck. View the star $G_2$ as a degenerate broom $B_{r+1,1}$ with diameter $d=1$ (so that the handle length is 0). Formula \eqref{eqn:join-broom-tip} yields $\Jmax(B_{r+1,1}) = r$, which is the correct value! This means that the argument for the previous case still applies when $G_2$ is the degenerate broom $B_{r+1,1}$. This completes the proof.
\end{proof}

\begin{lemma}
\label{lemma:min-join-double-broom-formulas}
Let $\Dnd$ be the balanced double broom with diameter $2 \leq d \leq n-1$.
Then 
\begin{equation*}
\Jmin(\Dnd) =
\begin{cases}
    (d-2) n^2 - (d^2 -2 d) n +  \frac{1}{3} (d^3-3d^2 + 2 d) & \mbox{if $n$ odd and $d$ odd,} \\
    (d-2) n^2 -(d^2 - 2 d - 1)n + \frac{1}{3} (d^3 -3d^2-d)+1 & \mbox{if $n$ odd and $d$ even,} \\
    (d-2)n^2 - (d^2-2d-2)n+\frac{1}{3} (d^3-3d^2-d) & \mbox{if $n$ even and $d$ odd,} \\
    (d-2)n^2 - (d^2-2d-1)n+\frac{1}{3} (d^3-3d^2 + 2d)-1 & \mbox{if $n$ even and $d$ even.}
\end{cases}
\end{equation*}
\end{lemma}

\begin{proof}
We handle the cases $d=2$ and $d=n-1$ exceptionally. Observe that $S_n=D_{n,2}$, and both formulas for even diameter $d=2$ simplify to $n-1=\Jmin(S_n)$. Considering $P_n=L_{n,n-1}$ the two formulas for order $n$ and diameter $n-1$ (with opposite parities) simplify to the formulas of Lemma \ref{lemma:path-min-join-time}.

Consider $3 \leq d \leq n-2$ and let $c$ be the barycenter of the balanced double broom $\Dnd$. Splitting the graph at the barycenter $c$ results in two broom graphs $G_1, G_2$. The minimum joining time of $\Dnd$ is the sum of the maximum joining times of these two subgraphs:
$\Jmin(\Dnd) = \Jmax(G_1) + \Jmax(G_2)$.
The structures of $G_1$ and $G_2$ depend on the parities of $n$ and $d$. We have
$$
\Jmin(D_{n,d}) =
\begin{cases}
\Jmax(B_{(n+1)/2, (d-1)/2}) + \Jmax(B_{(n+1)/2, (d+1)/2})
& \mbox{if $n$ is odd and $d$ is odd}, \\
2 \Jmax(B_{(n+1)/2, d/2})
& \mbox{if $n$ is odd and $d$ is even}, \\
\Jmax(B_{n/2, (d-1)/2}) + \Jmax(B_{n/2+1, (d+1)/2})
& \mbox{if $n$ is even and $d$ is odd}, \\
\Jmax(B_{n/2+1, d/2}) + \Jmax(B_{n/2, d/2})
& \mbox{if $n$ is even and $d$ is even}. 
\end{cases}
$$
We evaluate these formulas using Lemma \ref{lemma:broom-join-max}, and simplify.
\end{proof}

\begin{cor}
 \label{cor:bestmeet-double-broom-formulas}
Let $\Dnd$ be the balanced double broom with diameter $2 \leq d \leq n-1$.
Then $\Tbestmeet(\Dnd)$ is given by the following formulas, depending on the parities of $n$ and $d$:
\begin{equation*}
\Tbestmeet(\Dnd) =
\begin{cases}
\frac{1}{2} \left((d-2) n -d^2 +3 d-2 \right) +\frac{d^3-6 d^2+11 d-6}{6 (n-1)} 
& \mbox{if $n$ odd and $d$ odd,} \\
\frac{1}{2} \left( (d-2) n -d^2 +3 d - 1 \right) +\frac{d^3-6 d^2+8 d}{2(n-1)}
& \mbox{if $n$ odd and $d$ even,} \\
 \frac{1}{2} \left( (d-2) n - d^2 + 3 d \right) +\frac{d^3-6 d^2+8 d}{6 (n-1)}
 & \mbox{if $n$ even and $d$ odd,} \\
\frac{1}{2}\left( (d-2) n- d^2 + 3 d - 1 \right) +\frac{d^3-6 d^2+11 d-6}{6 (n-1)}
 & \mbox{if $n$ even and $d$ even.} \\
\end{cases}
\end{equation*}  
\end{cor}

\begin{proof}
For $G \in \Tnd$, we have $\Tbestmeet(G) = \Jmin(G)/2(n-1)$. The best meeting time formulas follow from Corollary \ref{cor:bestmeet-double-broom-formulas} and algebraic simplification.
\end{proof}

As with our minimizing structures, it is edifying to consider a sequence of balanced double brooms $\{ L_{n,d(n)} \}$ for increasing values of $n$, where $d=d(n)$ is a function of $n$. Using asymptotic notation, we have
$$
\Tbestmeet(\Dnd) = 
\begin{cases}
    \frac{d-2}{2} n + o(n) & \mbox{when } d \mbox{ is constant}, \\
     \frac{1}{2} n d(n) + o(n d(n)) & \mbox{when } d(n) = o(n) \mbox{ is sublinear}, \\
     \frac{3c-3c^2+c^3}{6} n  + o(n) & \mbox{when } d(n) = cn + o(n) \mbox{ where } 0 < c \leq 1. \\
\end{cases}
$$


With these formulas in hand, we can prove a surprising inequality: when $n \geq 9$ is odd, we have
$\Jmin(D_{n,n-2}) = \Jmin(B_{n,n-2}) > \Jmin(P_n) = \Jmin(D_{n,n-1})$. Even though the diameter of $D_{n,n-1}$ is smaller than that of $P_n$, it seems that having two barycenters (like $D_{n,n-1}$) is better than having one (like $P_n$).

\begin{cor}
\label{cor:min-join-dnd}
If $n$ is even then
$$
\max_{2 \leq d \leq n-1} \Jmin(\Dnd) = \Jmin(P_n) = \frac{1}{3} \left(n^3-3 n^2+5 n-3\right).
$$
If $n$ is odd then
$$
\max_{2 \leq d \leq n-1} \Jmin(\Dnd) = 
\begin{cases}
     \Jmin(P_n) = \frac{1}{3} \left(n^3-3 n^2+2 n\right) & \mbox{if $3 \leq n \leq 7$,} \\
    \Jmin(B_{n,n-2}) = \frac{1}{3} \left(n^3-3 n^2+5 n-24\right)  & \mbox{if $n \geq 9$.} \\    
\end{cases}
$$
\end{cor}

\begin{proof}
It is straightforward to confirm that each of the four formulas of Lemma \ref{lemma:min-join-double-broom-formulas} is increasing for $2 \leq d \leq n-1$.
For example, the derivative (with respect to $d$) of the equation for $n$ odd and $d$ odd is
$$
d^2-2(n-1) d+n^2+2 n-\frac{1}{3}
$$
which has critical points $n + 1 \pm 2\sqrt{3}/{3}$, so this function is increasing for $2 \leq d \leq n-1$. The other three cases are similar.

So we simply need to compare the values (using the correct formulas from Lemma \ref{lemma:min-join-double-broom-formulas}) for $D_{n,n-1}$ and $D_{n,n-2}$. When $n$ is even, we have
$\Jmin(D_{n,n-1}) - \Jmin(D_{n,n-2}) = 8$, so the maximum  is $\Jmin(D_{n,n-1}) = \Jmin(P_n)$.
Meanwhile, when $n$ is odd, we have
$\Jmin(D_{n,n-1}) - \Jmin(D_{n,n-2}) = 8-n$, so the maximum  is $\Jmin(D_{n,n-1}) = \Jmin(P_n)$ for $3 \leq n \leq  7$, and is $\Jmin(D_{n,n-2})=\Jmin(B_{n,n-2})$ for $n \geq 9$.

The formulas in the corollary are calculated directly from those in Lemma \ref{lemma:min-join-double-broom-formulas}, taking $d=n-1$ and $d=n-2$, as needed.
\end{proof}

We conclude this subsection by calculating the best meeting time for these extremal graphs.
\begin{cor}
\label{cor:max-bestmeet-formulas}
If $n$ is even then
$$
 \Tbestmeet(P_n) = \frac{1}{6} ( n^2 -2n +3).
$$
If $n$ is odd then
$$
 \Tbestmeet(P_n) = \frac{1}{6}(n^2-2n)
$$
and
$$
 \Tbestmeet(B_{n,n-2}) = \frac{1}{6}(n^2-2n+3) - \frac{4}{(n-1)}.
$$
\end{cor}

\begin{proof}
For $G \in \Tnd$, we have $\Tbestmeet(G) = \Jmin(G)/2(n-1)$. The best meeting time formulas follow from Corollary \ref{cor:min-join-dnd} and algebraic simplification.
\end{proof}

\subsection{Proof of Theorem \ref{thm:max-bestmeet}}

We prove Theorem \ref{thm:max-bestmeet} using a constructive lemma that alters a tree in $\Tnd$ until we arrive at the balanced double broom. With each change, the joining time strictly increases. 

We will often consider subgraphs that are brooms, so it will be convenient to introduce some notation beforehand.
By Lemma \ref{lemma:broom-join-max}, the maximum joining time for the broom $\Bnd$ is achieved by the handle leaf $v_d$. We use the formula of Lemma \ref{lemma:broom-join-max} to define three differences. First,
we set 
\begin{equation}
\label{eqn:big-delta-plus}
\Delta^+(\broom{n}{d}) = \Jmax(\broom{n+1}{d+1}) - \Jmax(\broom{n}{d})
= 4n^2-4n+1
\end{equation}
to be the change in $\Jmax(\broom{n}{d})$ when we increase the handle length of $\broom{n}{d}$ by one.  
Next, we set
\begin{equation}
\label{eqn:delta-plus}
\delta^+(\broom{n}{d}) = \Jmax(\broom{n+1}{d}) - \Jmax(\broom{n}{d})
= 4(d-1)(2n-d))+1
\end{equation}
to be the change in $\Jmax(\broom{n}{d})$ when we add a bristle leaf to $\broom{n}{d}$.  
Finally, we set $\delta^-(\broom{n}{d})$ to be the (negative) change in $\Jmax(\broom{n}{d})$ when we subtract a leaf from $\broom{n}{d}$. There are two cases, depending on the number of bristles: for
 $2 \leq d \leq n-2$, we have
\begin{equation}
\label{eqn:delta-minus-broom}
\delta^-(\broom{n}{d})  
= - \delta^+(\broom{n-1}{d})
= -4 (d - 1) (2 (n - 1) - d) - 1
\end{equation}
 and for $d=n-1$, equation \eqref{eqn:path-join-time} yields
\begin{align}
\nonumber
\delta^-(\broom{n}{n-1})
&=
 \Jmax(\broom{n-1}{n-2}) - \Jmax(\broom{n}{n-1}) \\
\label{eqn:delta-minus-path}
& =
 \Jmax(P_{n-1}) - \Jmax(P_{n})
 = -(2n-3)^2.
\end{align}

Three elementary observations that we will need later are
\begin{equation}
\label{eqn:delta1}
  \delta^+(\broom{n+1}{d}) > \delta^+(\broom{n}{d})  
\end{equation}
and
\begin{equation}
\label{eqn:delta2}
  \delta^+(\broom{n}{n-1}) > \delta^+(\broom{n-1}{n-2})  
\end{equation}
and
\begin{equation}
\label{eqn:delta3}
\Delta^+(\broom{n}{d}) > \delta^+(\broom{n}{d}) > - \delta^-(\broom{n}{d}).
\end{equation}

 Let $c$ be a barycenter of $G=(V,E)$, where $\deg(c)=k$.
For the remainder of this section, let  $ \{G_1, G_2, \ldots, G_k \}$ be the $c$-split of $G$, and let  $\J_i(c)$ denote the joining time to $c$ in tree $G_i$. Recall from equation \eqref{eqn:join-decomp} that $\J(c) = \sum_{i=1}^k  \J_{i} (c).$

\begin{lemma}
\label{lemma:to-double-broom}
If $G \in \Tnd$  is not a double broom then there exists $\newG \in \tree{n}{d'}$ for some $d' \leq d$  such that $\newG$ is a double broom and 
$\Jmin(G) < \Jmin(\newG).$ 
\end{lemma}

\begin{proof}
    
Let $G \in \Tnd$ with barycenter $c$.
Suppose that $G$ is not a double broom.
We use a three-phase process to relocate the vertices and edges of $G$ to increase  $\min_{v \in V} \J(v)$.

{\bf Phase One.} Suppose that some $G_i$ is neither a path nor a broom with handle leaf $c$.
By Lemma \ref{lemma:broomify},  we can increase $J_i(c)$ by replacing $G_i$ with the broom $B_{n_i, r_i}$ where $n_i = |V_i|$ and $r_i = \max_{v \in V_i} d(v,c)$. The resulting graph $\newG$ is also a graph in $\tree{n}{d}$. Furthermore, $c$ is the barycenter of $\newG$ and  $\newJ(c) > \J(c)$. Repeat this process until every $G_i$ is a broom. If $k=2$ then the lemma is proved. Otherwise, move on to phase two.

{\bf Phase Two.} We have a tree $G \in \tree{n}{d}$ with barycenter $c$ with $\deg(c)=k>2$ such that each $G_i = B_{n_i, r_i}$  is either a path or a broom with handle leaf $c$. 
Without loss of generality, the subtrees  $G_1,G_2, \ldots, G_k$ are ordered so that
\begin{equation}
\label{eqn:delta-top-two}
   \min \{ \delta^+(G_1) , \delta^+(G_2) \} \geq  \max \{ \delta^+(G_3), \delta^+(G_4), \ldots   \delta^+(G_k) \}.     
\end{equation}
 This ordering will allow us to move the bristle leaves (one by one) from $G_i$, where $3 \leq i \leq k$,  to brooms $G_1$ and $G_2$, while also increasing the value of $J(c)$. Furthermore, we will distribute these vertices so that $c$ remains the barycenter of the graph. If $r_1+r_2=d$, then there is a geodesic path in $G_1 \cup G_2$, and we move directly to phase three. 
 
 Otherwise, the goal of phase two is to increase the diameters of graphs $G_1, G_2$ until $r_1+r_2=d$. (Note that we may temporarily increase the overall diameter of the graph, but this will be rectified in phase three.) Each time that we move a leaf, there are two cases. 

 \emph{Case 1:} $|V(G_1)| < \lceil n/2 \rceil+1$. Pick any $G_j$ for $3 \leq j \leq k$.  If $G_j$  is a broom with at least two bristle leaves, then we move a bristle leaf  $w$ from $G_j$ to expand the handle of $G_1$.  Otherwise,  $G_j$  is a path. We move the leaf  $w \neq c$ from $G_j$ to expand the handle of $G_1$. 
 In either circumstance, we have
\begin{equation}
\label{eqn:phase-two}
  \Delta^+(G_1) \geq   \delta^+(G_j) \geq -  \delta^-(G_i)   
\end{equation}
 by equations \eqref{eqn:delta3}  and \eqref{eqn:delta-top-two},
 and therefore the net change to $J(c)$ is nonnegative.
 Finally, equations \eqref{eqn:delta1} and \eqref{eqn:delta2}  guarantee that equation \eqref{eqn:delta-top-two} still holds for the updated tree. Therefore we can repeat this process on the updated tree.

\emph{Case 2:} $|V(G_1)| = \lceil n/2 \rceil+1$. We are guaranteed that $|V(G_2)| < \lceil n/2 \rceil+1$ (because $\deg(c)=k\geq 3)$. The argument is the analogous to the previous case, but we move the leaf to $G_2$. This maintains $c$ as the barycenter of the tree.

We repeatedly move a leaf $w \neq c$ from some $G_j$ until $r_1+r_2 = d$. We then move on to phase three.

{\bf Phase Three.} We know that $r_1+r_2=d$ and that equation \eqref{eqn:delta-top-two} holds. Now we repeatedly move a leaf $w  \neq c$ from $G_3, G_4 \ldots G_k$, to become a bristle leaf for one of $G_1, G_2$. As in phase two, we choose the target tree so that $c$ remains the barycenter of tree. The argument is identical to the two cases for Phase Two, except equation \eqref{eqn:phase-two} is replaced by
 $$
 \delta^+(G_1) \geq   \delta^+(G_j) \geq -  \delta^-(G_i).
 $$

  We repeatedly move a leaf from each $G_i$ for $3 \leq i \leq k$ until each of these subgraphs only contains the barycenter $c$. At this point, we have a double broom with diameter $d$.
\end{proof}

 We can now prove Theorem \ref{thm:max-bestmeet}.

\begin{proof}[Proof of Theorem \ref{thm:max-bestmeet}]
Let $G \in \Tnd$ be a tree that is not a double broom. By Lemma \ref{lemma:to-double-broom}, there is a double broom $\newG \in \tree{n}{d}$  such that $\Jmin(G) < \Jmin(\newG).$ By Lemma \ref{lemma:balanced-double-broom}, the balanced double broom $D_{n,d}$ achieves $\max_{\newG \in \mathcal{D}_{n,d}} \Jmin(\newG)$. So we finally conclude that
$
\Jmin(G) < \Jmin(\newG) \leq \Jmin(D_{n,d}),
$
with equality if and only if $\newG = \Dnd$.
So the unique tree that achieves $\max_{G \in \Tnd} \Tbestmeet(G)$ is the balanced double broom $\Dnd$.

The formulas for $\Tbestmeet(\Dnd)$ were established in Corollary \ref{cor:bestmeet-double-broom-formulas}.
\end{proof}

\section{The best meeting time for trees of fixed order}

This brief section contains the proof of Theorem \ref{thm:best-meet-trees}. We aggregate our results and give the upper and lower bounds for $\Tbestmeet(G)$ when $G \in \Tn$ is a tree on $n$ vertices.

\begin{proof}[Proof of Theorem \ref{thm:best-meet-trees}]
Let $G \in \Tn$. First, we consider the lower bound. The formulas of Theorem \ref{thm:min-bestmeet} reveal that $\Tbestmeet(G) \geq 1/2$, with equality if and only if $G=S_n$ is the star graph. 

Now we turn to the upper bound. Theorem \ref{thm:max-bestmeet} guarantees that the maximizer of $\Tbestmeet(G)$ for $G \in \Tnd$ is a balanced double broom. Corollary \ref{cor:min-join-dnd} determines the maximizing tree among all balanced double brooms of order $n$. When $n$ is even, the maximizer is the path $P_n$. When $n$ is odd, the maximizer is the path $P_n$ for $3 \leq n \leq 7$, and it is the balanced double broom $D_{n,n-2}$ for $n \geq 9$. The formulas for the meeting times of these graphs are given in Corollary \ref{cor:max-bestmeet-formulas}.
\end{proof}

\section{Conclusion}

We have characterized the extremal values of $\Tbestmeet(G) = \min_{w \in V} \sum_{v \in V} \pi_v H(v, w)$ for graphs $G \in \Tnd$ on $n$ vertices with diameter $d$. The maximizing tree is the double broom graph $D_{n,d}$, while the  minimizing structure is the balanced lever graph $L_{n,d}$. 

Besides Kemeny's constant, there are other mixing measures to investigate for the family $\Tnd$ of trees of order $n$ with diameter $d$.
For example,  a variant of the mean first passage time is
$ \frac{1}{n(n-1)} \sum_{u \in V} \sum_{v \in V} H(u,v)$, which is similar to Kemeny's constant. Perhaps techniques from \cite{MaWang2020} could be useful in determining the best and worst trees for this quantity. 

Mixing measures that capture the convergence rate to the stationary distribution are of particular interest. For example,  the maximum relaxation time for graphs on $n$ vertices was characterized in \cite{AKSOY2018}. It would be interesting to determine the minimum and maximum relaxation times for trees of order $n$ and diameter $d$.
Meanwhile, \cite{BHOV} showed that the balanced double broom $D_{n,d}$ achieves the maximum \emph{exact mixing time} $\max_{v \in V} H(v,\pi)$, where $H(v, \pi) = \max_{w \in V} H(w,v) - H(\pi,v)$. Characterizing the \emph{best mixing time} $\min_{v \in V} H(v,\pi)$ for trees in $\Tnd$ remains an open question. 

\bibliographystyle{plain}
\bibliography{mybib}

\end{document}